\documentclass[11pt, a4paper]{amsart}

\usepackage{geometry} 
\usepackage{fancyhdr}
\usepackage{url}
\usepackage{amssymb}
\usepackage{amsmath}
\usepackage{accents}
\usepackage{eucal}
\usepackage{amscd}
\usepackage[shortalphabetic]{amsrefs}
\usepackage{times}

\allowdisplaybreaks[1]

\newtheorem{thm}{Theorem}[section]
\newtheorem{lem}[thm]{Lemma}
\newtheorem{prop}[thm]{Proposition}

\numberwithin{equation}{section}

\providecommand{\abs}[1]{\lvert#1\rvert}

\newcommand{\p}{\partial}

\DeclareMathOperator{\osc}{osc}

\author{Charles Baker}\thanks{Research partially supported by Discovery Grant DP0985802 from the Australian Research Council and partially by The Leverhulme Trust}
\email{rogercharlesbaker@gmail.com}
\title[A proof of the parabolic Schauder estimates]{A proof of the parabolic Schauder estimates using Trudinger's method and the mean value property of the heat equation}
\begin{document}
\maketitle

\section{Introduction}
One method available to prove the Schauder estimates is Neil Trudinger's method of mollification (\cite{Tru}). In the case of second order elliptic equations, the method requires little more than mollification and the solid mean value inequality for subharmonic functions. The method was lated adapted to the parabolic setting by Xu-Jia Wang in \cite{Wan}, however in that presentation Wang uses an auxiliary estimate coming from the fundamental solution of the heat equation and the mean value of property of subsolutions of the heat equations is not used.  Our goal in this article is show how the mean value property of subsolutions of the heat equation can be used in a similar fashion as the solid mean value inequality for subharmonic functions in Trudinger's original elliptic treatment, providing a relatively simple derivation of the interior Schauder estimate for second order parabolic equations.

\section{Preliminaries}
For an open subset $U \in \mathbb{R}^d$, we denote the corresponding open parabolic domain $U \times (0, T) \subset \mathbb{R}^{d+1}$ by $U_T$. We denote a (backwards) parabolic cylinder by $Q_R = B_R \times (t - R^2, t)$ . We often notate a point $(x,t) \in U_T$ by $X$. A (second order parabolic) mollifier is a fixed smooth function $\rho \in C_c^{ \infty }( \mathbb{R}^{d+1})$ with $\iint_{ \mathbb{R}^{d+1} } \rho \, dx \, dt = 1$.  For $\tau > 0$ we define the scaled mollifier
	\begin{equation}
		\rho_{\tau}(x,t) := \frac{1}{\tau^{d+2}} \rho\left(\frac{x}{\tau},\frac{t}{\tau^{2}}\right).
	\end{equation}
Let  $U \in \mathbb{R}^{d+1}$ and $u \in L^1_{\text{loc}}(U)$.  For $0 < \tau < d(X, \p P)$, the mollification of $u$ is given by
	\begin{equation*}
		u_{\tau}(x,t) := \frac{1}{\tau^{d+2}} \iint \rho \left(\frac{x-y}{\tau},\frac{t-s}{\tau^{2}} \right) u(y,s) \, dy \, ds
	\end{equation*}
and satisifes spt $u_{\tau} \subset U_{\tau}$, where $U_{\tau} = \{ X \in U : d(X, \p U) > \tau \}$.

The parabolic distance between two points $X = (x,t)$ and $Y = (y,s)$ is defined to be
\[ d(X, Y) := \max \{ \abs{ x-y}, \abs{t-s}^{1/2} \}. \]
We use both $\sup$ and $\abs{ \cdot }_0$ to denote the supremum of a function. The H\"{o}lder seminorm is defined by
\[ [ u ]_{\alpha; U_T} : = \sup_{X \neq Y \in U_T} \frac{ \abs{ u(X) - u(Y) } }{ d(X,Y)^{\alpha} } \]
and the H\"older norms by
\begin{align*}
	&\abs{ u }_{2, 1; \, U_T } := \sum_{ k=0 }^{2} \abs{ \p_x^k u }_{ 0; \, U_T } + \abs{ \p_t u }_{ 0; \, U_T } \\
	&\abs{ u }_{2, 1, \alpha; \, U_T } := \abs{ u }_{2, 1; \, U_T } +  [ \p_x^{2} u ]_{ \alpha; \, U_T } + [ \p_t u ]_{ \alpha; \, U_T } \\
\end{align*}	
The set of functions
\begin{equation*}
	\{ u \in C^{2, 1}(U_T) : [u]_{ 2, 1, \alpha; \, U_T } < \infty \}
\end{equation*}
endowed with the norm $\abs{ u }_{ 2m, 1, \alpha; \, U_T }$ is called a H\"{o}lder space.  Written out in full the norm is
\begin{equation*}
	\abs{ u }_{ 2, 1, \alpha; \, U_T } := \sum_{ k=0 }^{2} \abs{ \p_x^k u }_{ 0; \, U_T } + \abs{ \p_t u }_{ 0; \, U_T }+ [ \p_x^{2} u ]_{ \alpha; \, U_T } + [ \p_t u ]_{ \alpha; \, U_T }.
\end{equation*}
We sometimes use
\[ \p^{2,1} := \p_x^2 + \p_t, \]
e.g,
\[ [ \p^{2,1}u ]_{\alpha; U_T} = [ \p_x^2 u ]_{\alpha; U_T}  + [ \p_t u ]_{\alpha; U_T}, \]
to compress notation a little.

We begin by establishing some very basic estimates that will be use throughout.
\begin{prop}
We have $u_{\tau} \in C_c^{\infty}$.
\end{prop}

\begin{prop}
Let $u \in L^1_{\text{loc}}(U)$.  The following estimates hold:
	\begin{align}
		&\abs{ u_{ \tau } }_{0; \, U_{\tau} } \leq \abs{u}_{ 0; \, U_{\tau} } \label{e: Holder est 1} \\
		&\abs{ \p_x^i \p_t^j u_{\tau}(x,t) }_{ 0; \, U_{\tau} } \leq C \tau^{-i-2j} \abs{u}_{0; \, U_{\tau} }. \label{e: Holder est 2}
	\end{align}
\end{prop}
\begin{proof}
To prove \eqref{e: Holder est 1}, we have
\begin{align*}
	u_{ \tau }(x,t) &= \frac{1}{ \tau^{d+2} } \iint \rho \left( \frac{x-y}{\tau},\frac{t-s}{ \tau^{2}} \right) u(y,s) \, dy \, ds \\
	&\leq \abs{ u }_{0; \, U_{\tau} } \cdot \frac{1}{ \tau^{d+2} } \iint \rho \left( \frac{x-y}{\tau},\frac{t-s}{ \tau^{2}} \right) \, dy \, ds \\
	&= \abs{ u }_{0; \, U_{\tau} }.
\end{align*}
And for \eqref{e: Holder est 2}:
\begin{align*}
	\p_x^i \p_t^j u_{\tau}(x,t) &= \frac{1}{ \tau^{d+2} } \iint_{U_{\tau}} \p_x^i \p_t^j \rho \left( \frac{x-y}{\tau},\frac{t-s}{ \tau^{2}} \right) u(y,s) \, dy \, ds \\
	&\leq C \tau^{-i-2j}\abs{ u }_{ 0; \, U_{\tau} }.
\end{align*}
\end{proof}

\begin{prop}
Let $u \in C^{\alpha}_{\text{loc}}(U)$.  The following estimates hold:
	\begin{align}
		&\abs{ u_{\tau}(x,t) - u(x,t) }_{ 0; \, U_{\tau} }\leq \tau^{\alpha} [u]_{\alpha; \, U_{\tau} } \label{e: Holder est 3} \\
		&\abs{ \p_x^i \p_t^j u_{\tau}(x,t) }_{ 0; \, U_{\tau} } \leq C \tau^{ \alpha -i-2j } [u]_{\alpha; \, U_{\tau} }. \label{e: Holder est 4}
	\end{align}
\end{prop}

\begin{proof}
For estimate \eqref{e: Holder est 3} we have
\begin{align*}
	u_{\tau}(x,t) - u(x,t) &= \frac{1}{ \tau^{d+2} } \iint \rho \left( \frac{x-y}{\tau},\frac{t-s}{ \tau^{2}} \right) (u(y,s) - u(x,t)) \, dy \, ds \\
	&\leq \osc_{U_{\tau} } u \\
	&\leq \tau^{\alpha} [u]_{\alpha; U_{\tau} }.
\end{align*}
To prove the second estimate we have
\begin{align*}
		\p_x^i \p_t^j u_{\tau}(x,t) &= \frac{1}{ \tau^{d+2} } \iint \p_x^i \p_t^j \rho \left( \frac{x-y}{\tau},\frac{t-s}{ \tau^{2}} \right) u(y,s) \, dy \, ds \\
		&=  \frac{1}{ \tau^{d+2} } \iint \p_x^i \p_t^j \rho \left( \frac{x-y}{\tau},\frac{t-s}{ \tau^{2}} \right) (u(y,s) - u(x,t)) \, dy \, ds \\
		&\quad +  \frac{u(x,t)}{ \tau^{d+2} } \iint \p_x^i \p_t^j \rho \left( \frac{x-y}{\tau},\frac{t-s}{ \tau^{2}} \right) \, dy \, ds.
\end{align*}
The mollifier $\rho$ is has compact support on $U_{\tau}$ and so the last term vanishes by the Divergence Theorem.  Continuing, we have
	\begin{align*}
		\p_x^i \p_t^j u_{\tau}(x,t) &= \frac{1}{ \tau^{d+2} } \iint_{U_{\tau}} \p_x^i \p_t^j \rho \left( \frac{x-y}{\tau},\frac{t-s}{ \tau^{2}} \right) (u(y,s) - u(x,t)) \, dy \, ds \\
		&\leq C \tau^{-i-2j}\osc_{Q_\tau} u \\
		&\leq C \tau^{ \alpha -i-2j } [u]_{ \alpha; \, U_{\tau} }.
	\end{align*}
\end{proof}

\section{The interior elliptic Schauder estimate for functions of compact support}

To motivate things a little for the parabolic setting, we first briefly show how Trudinger's method works in the elliptic realm by treating the Poisson equation.  The crucial ingredient in Trudinger's method is the following norm equivalence:

\begin{lem}
Let $U \subset \mathbb{R}^d$ be an open bounded domain and $u \in C^{ \alpha }(U)$, where $\alpha \in (0,1)$. Let $R > 0$ and $\tau_0$ be small constants, both of which will be fixed in the course of the proof. There exists a constant $C = C(d, \alpha)$ such that the norm equivalence
\begin{equation*}
	\frac{1}{C}[u]_{ \alpha; \, B_{R} } \leq \sup_{ 0 < \tau < \tau_0  } \tau^{1-\alpha} \abs{\p_x u_{\tau}}_{ 0; \, B_{ R } } \leq C[u]_{ \alpha; \, B_{R} }.
\end{equation*}
is valid.
\end{lem}
\begin{proof}
The inequality on the right follows directly from equation \eqref{e: Holder est 4} (the elliptic version) by choosing the appropriate values for the indices $i$: choosing $i = 1$ (there is no $j$ in the elliptic mollifier) gives
\begin{equation*}
	\abs{ \p_x u_{\tau}}_{ 0; \, B_{ R } } \leq C \tau^{ \alpha - 1 } [u]_{\alpha; \, B_{ R } }.
\end{equation*}
The first inequality requires a little more work.  Let $x, y \in U$ satisfy $\abs{x-y} < d_x/2$ and let $\tau$ satisfy $0 < \tau <\tau_0 < d_x/2$, where $\tau_0$ will be fixed shortly. For $ \abs{ x - y } < R = d_x/2 $, by the triangle inequality
\begin{align*}
		\abs{u(x) - u(y)} &\leq \abs{u(x) - u_{\tau}(x)} + \abs{ u_{\tau}(x) - u_{\tau}(y) } + \abs{ u_{\tau}(y) - u(y) } \\
		&\leq 2\tau^{\alpha}[u]_{ \alpha; \, B_{ R }} + \abs{ \p_x u_{\tau} }_{ 0; \, B_{ R } }\abs{x-y}.
\end{align*}
Set $\tau = \epsilon \abs{x-y}$, where $\epsilon < 1/2$.  Factoring out and dividing by $\abs{x-y}^{\alpha}$ we find
\begin{equation*}
	 \abs{ u(x) - u(y) } \leq \abs{x-y}^{\alpha} \left( 2\epsilon^{\alpha}[u]_{\alpha; B_R} +  \epsilon^{\alpha - 1}\tau^{1-\alpha} \abs{ \p_x u_{\tau} }_{0; \, B_R} \right).
\end{equation*} 
Choose $\epsilon = 4^{-1/\alpha}$ then set $\tau_0 = 4^{-1/\alpha} . \,  d_x/2$. Taking the supremum over $\tau \in (0, \tau_0)$ completes the proof.
\end{proof}

We now derive the interior Schauder estimate for Poisson's equation. For simplicity, we only consider solutions with compact support in a ball $B_R$. The method extends in the usual way to more general equations and domains by using cutoff functions and Simon's absorption lemma (see, for example, \cite{Sim}).  Suppose that $u \in C_c^{2, \alpha}(B_R)$ solves
\begin{equation*}
	-a^{ij}(x)\p_{ij}u(x) = f(x),
\end{equation*}
where we assume $a^{ij}, f \in C^{\alpha}(B_R)$, and there are constants $\lambda >0$ and $\Lambda < \infty$ such that $\lambda \abs{\xi}^2 \leq a^{ij}\xi_i\xi_j \leq \Lambda \abs{ \xi }^2$.  We proceed by the method of freezing coefficients, and accordingly fix a point $x_0 \in \mathbb{R}^d$ a rewrite the above equation equation as
\begin{align}
\notag	-a^{ij}(x_0)\p_{ij}u(x) - f (x_0)&= (a^{ij}(x) - a^{ij}(x_0) )\p_{ij}u(x) + f(x) - f(x_0)\\
	&:= g(x) \label{eqn: Poisson 1}
\end{align}
After a coordinate transformation, we can assume without loss of generality that $a^{ij}(x_0) = \delta^{ij}$, that is we can assume \eqref{eqn: Poisson 1} is the Poisson equation. Mollify equation \eqref{eqn: Poisson 1} to get
\begin{equation*}
	-\Delta u_{\tau} - f(x_0) = g_{\tau}
\end{equation*}
and then differentiate thrice with respect to $x$ to obtain
\begin{equation*}
	-\Delta \p_x^3 u_{\tau} = \p_x^3 g_{\tau}.
\end{equation*}
Using inequality \eqref{e: Holder est 4} we can estimate
\begin{align*}
	\abs{ \p_x^3 g_{\tau} }_{0; \, B_R } &\leq C(d) \tau^{-3} \abs{ g }_{0; \, B_R } \\ 
	&\leq C(d) \tau^{-3} R^{\alpha} \big(  [ a ]_{\alpha; \, B_R } \abs{ \p_x^2 u }_{0; \, B_R } + [ f ]_{\alpha; \, B_R } \big).
\end{align*}
Recall the solid mean value inequality for subharmonic functions: If $v$ solves $-\Delta v(x) \leq 0$ on a ball $B_R(x) \subset \mathbb{R}^d$, then $v$ satisfies
\[ v(x) \leq \frac{ C(d) }{ R^n } \int_{B_R} v(y) \, dy. \]
To apply this inequality to our situation, noting $\Delta \abs{ x }^2 = 2d$, we have
\[	-\Delta \left( \p_x^3 u_{\tau} + \frac{ \abs{ \p_x^3 g_{\tau} }_{0; \, B_R } \abs{ x }^2 }{ 2n } \right) = -\Delta \p_x^3 u_{\tau} - \abs{ \p_x^3 g_{\tau}  }_{0; \, B_R} \leq 0. \]
Thus the function $\p_x^3 u_{\tau} + \abs{ \p_x^3 g_{\tau} }_{0; \, B_R } \abs{ x }^2 / (2d)$ is subharmonic and applying the mean value inequality and estimating we obtain
\begin{align*}
	\abs{ \p_x^3 u_{\tau}(x_0) } &\leq C(d) \left( \frac{1}{R^d} \left\lvert \int_{B_R} \p_y^3 u_{\tau}(y)  \, dy \right\rvert +  R^2 \abs{ \p_x^3 g_{\tau} }_{0; \, B_R } \right) \\
	&\leq C(d) \left( \frac1R \osc_{ B_R } \p_x^2 u_{\tau}(x) + \tau^{-3} R^{2+\alpha} \big( [ a ]_{\alpha; B_{R} } \abs{ \p_x^2 u }_{0; \, B_{R} } + [ f ]_{\alpha; \, B_{R} } \big) \right) \\
	&\leq C(d) \Big( R^{ \alpha - 1 } [ \p_x^2 u ]_{ \alpha; \, B_R } + \tau^{-3} R^{2+\alpha} \big( [ a ]_{\alpha; B_{R} } \abs{ \p_x^2 u }_{0; \, B_{R} } + [ f ]_{\alpha; \, B_{R} } \big) \Big).
\end{align*}
We set $R = N\tau$ and return to the original coordinates to find
\[ \tau^{1-\alpha}\abs{ \p_x^3 u_{\tau}(x_0) } \leq C(d, \lambda, \Lambda) \left( N^{ \alpha - 1 } [ \p_x^2 u ]_{ \alpha; \, B_R } + N^{2+\alpha} \big( [ a ]_{\alpha; B_{R} } \abs{ \p_x^2 u }_{0; \, B_{R} } + [ f ]_{\alpha; \, B_{R} } \big) \right), \]
then taking the supremum over $\tau$ and using the norm equivalence we obtain
\[ [ \p_x^2 u ]_{\alpha; \, B_R } \leq C(d, \lambda, \Lambda, \alpha) \left( N^{ \alpha - 1 } [ \p_x^2 u ]_{ \alpha; \, B_R } + N^{2+\alpha} \big( [ a ]_{\alpha; B_R } \abs{ \p_x^2 u }_{0; \, B_R } + [ f ]_{\alpha; \, B_R } \big) \right). \]
Choosing $N$ sufficiently large and using the H\"{o}lder space interpolation inequality on the right gives the desired estimate, namely
\[ [ \p_x^2 u ]_{\alpha; \, B_R} \leq C \big([ f ]_{\alpha; \, B_R } + \abs{ u }_{0; \, B_R } \big), \]
where $C$ depends on $d, \lambda, \Lambda$, and $\alpha$.

\section{The interior parabolic Schauder estimate for functions of compact support}
Having given a feel for Trudinger's method, we move on to use this method to derive the Schauder estimates for second order parabolic equations.  The crucial equivalence of norms lemma in the parabolic setting is the following:

\begin{lem}
Let $U \subset \mathbb{R}^{d+1}$ be an open bounded domain and $u \in C^{ \alpha }(U)$, where $\alpha \in (0,1)$, and  $R > 0$ and $\tau_0$ be small constants, both of which will be fixed in the course of the proof. There exists a constant $C = C(d, \alpha)$ such that the norm equivalence
\begin{equation*}
		\frac{1}{C}[u]_{ \alpha; \, Q_{R} } \leq \sup_{0 < \tau < \tau_0 } \left\{ \tau^{1-\alpha} \abs{\p_x u_{\tau}}_{ 0; \, Q_{ R } } + \tau^{2-\alpha}\abs{\p_t u_{\tau}}_{0; Q_{R} } \right\} \leq C[u]_{ \alpha; \, Q_{R} }.
	\end{equation*}
is valid.
\end{lem}
\begin{proof}
The second inequality follows directly from equation \eqref{e: Holder est 4} by choosing the appropriate values for the indices $i$ and $j$.  To prove the spatial part of the second inequality, choosing  $i = 1$ and $j = 0$ in estimate \eqref{e: Holder est 4} gives
\begin{equation*}
	\abs{ \p_x u_{\tau}(x,t) }_{ 0; \, Q_{ R } } \leq C \tau^{ \alpha - 1 } [u]_{\alpha; \, Q_{ R } }.
\end{equation*}
The temporal estimate follows similarly. Moving on to the proof of the first inequality, let $X,Y \in U_T $ be points such that $d(X,Y) < d_X/2$. To simplify notation a little, set $R = d_X/2$; thus $d(X,Y) \in Q_R \subset U_T$. Let $\tau$ satisfy $0 < \tau <\tau_0 < R$, where $\tau_0$ will be fixed shortly. For $d(X,Y) < R$,

\begin{align*}
		\abs{u(X) - u(Y)} &\leq \abs{u(X) - u_{\tau}(X)} + \abs{ u_{\tau}(Y) - u(Y) } + \abs{ u_{\tau}(x,t) - u_{\tau}(y,t) } + \abs{ u_{\tau}(y,t) - u_{\tau}(y,s) } \\
		&\leq 2\tau^{\alpha}[u]_{ \alpha; \, Q_{ R }} + \abs{x-y} \abs{ \p_x u_{\tau} }_{ 0; \, Q_{ R } } + \abs{t-s} \abs{ \p_t u_{\tau} }_{ 0; \, Q_{ R } }.
\end{align*}
Set $\tau = \epsilon d(X,Y)$, where $\epsilon < 1/2$.  Factoring out $d(X,Y)^{\alpha}$ we have
\begin{equation*}
	\abs{u(X) - u(Y)} \leq d(X,Y)^{\alpha} \left( 2\epsilon^{\alpha}[u]_{\alpha; \, Q_{ R }} + \epsilon^{\alpha-1}\tau^{1-\alpha}\abs{ \p_x u_{\tau} }_{ 0; \, Q_{ R } } + \epsilon^{\alpha-2}\tau^{2-\alpha}\abs{ \p_t u_{\tau} }_{ 0; \, Q_{ R } } \right).
\end{equation*}
Choose $\epsilon = 4^{-1/\alpha}$ then set $\tau_0 = 4^{-1/\alpha} \cdot \,  d_X/2$. Taking the supremum over $\tau \in (0, \tau_0)$ completes the proof.
\end{proof}
We now proceed similarly to Poisson's equation to derive the Schauder estimate for the nonhomongeneous heat equation. Again for simplicity, we only condisider solutions with compact support in a parabolic cylinder $Q_R$ as the more general estimates follow from this case using cutoff functions and Simon's adsorption lemma. Suppose that $u \in C_c^{2, \alpha}(Q_R)$ solves
\[ \p_t u(x,t) - a^{ij}(x,t)\p_{ij}u(x,t) = f(x,t), \]
where we assume $a^{ij}, f \in C^{\alpha}(Q_R)$, and there are constants $\lambda >0$ and $\Lambda < \infty$ such that $\lambda \abs{\xi}^2 \leq a^{ij}\xi_i\xi_j \leq \Lambda \abs{ \xi }^2$.    Again we freeze coefficients at a point $(x_0, t_0) \in Q_R$ and perform a coordinate transformation if necessary to get
\begin{align*}
	\p_t u(x,t) - \Delta u(x,t) - f(x_0, t_0) &= ((a^{ij}(x,t) - a^{ij}(x_0,t_0) ) \p_{ij}u(x,t) + f(x,t) - f(x_0,t_0) \\
	&:= g(x,t),
\end{align*}
and then mollify the equation to obtain
\begin{equation}
	\p_t u_{\tau}(x,t) - \Delta u_{\tau}(x,t) = g_{\tau}(x,t). \label{eqn: heat subsol 0}
\end{equation} 
Given the form of the norm equivalence, the desired Schauder estimate will follow if we can establish the estimates (for the spatial component of the Schauder estimate)
\begin{align}
	\abs{ \p_x^3 u_\tau(x, t) } &\leq C \left( \frac1R \osc_{ Q_R } \p_x^2 u + R^2\abs{ \p_x^3 g_{\tau} }_{0; \, Q_R} \right) \label{eqn: heat subsol 1} \\
	\abs{ \p_t\p_x^2 u_\tau(x, t) } &\leq C \left( \frac{1}{R^2} \osc_{ Q_R } \p_x^2 u + R^2\abs{ \p_t\p_x^2 g_{\tau} }_{0; \, Q_R} \right), \label{eqn: heat subsol 2}
\end{align}
and for the temporal part
\begin{align}
	\abs{ \p_x\p_t u_\tau(x, t) } &\leq C \left( \frac1R \osc_{ Q_R } \p_t u + R^2\abs{ \p_x\p_t g_{\tau} }_{0; \, Q_R} \right) \label{eqn: time subsol 1} \\
	\abs{ \p_t^2 u_\tau(x, t) } &\leq C \left( \frac{1}{R^2} \osc_{ Q_R } \p_t u + R^2\abs{ \p_t^2 g_{\tau} }_{0; \, Q_R} \right).  \label{eqn: time subsol 2}
\end{align}
We show how to obtain estimate \eqref{eqn: heat subsol 1} as the other three estimates follow in the same way. Recall the mean value property for subsolutions of the heat equation:  If $v$ is a subsolution to the heat equation on $\mathbb{R}^{d+1}$, that is if $v$ solves $\p_v - \Delta v \leq 0$, then $v$ satisfies
\begin{equation*}
	v(x, t) \leq \frac{1}{4r^n}\iint_{E(x,t;\,r)} v(y,s) \frac{ \abs{ x - y }^2 }{ (t-s)^2 } \, dy \, ds
\end{equation*}
for each $E(x,t; r) \subset \mathbb{R}^{d+1}$, where the heat ball $E(x,t;r)$ is the set given by $E(x,t;r) = \{ (y,s) \in \mathbb{R}^{d+1} : \abs{x-y}^2 \leq \sqrt{ -2\pi s \log[r^2/(-4\pi s)] }, s \in (t - r^2/(4\pi s), t) \}$. The radius of the heat ball $\sqrt{ -2\pi s \log[r^2/(-4\pi s)] }$ is often denoted by $R_r(s)$. We refer the interested reader to \cite{Eva} and \cite{Eck} for more information on the mean value property and heat balls.  Let us now show \eqref{eqn: heat subsol 1}:  Differentiate \eqref{eqn: heat subsol 0} thrice in space.  Since $\abs{ \p_x^3 g_{\tau} }_{0; \, E}\abs{x}^2/(2n)$ is independent of time we see
\begin{align*}
	\p_t \left( \p_x^3 u_{\tau} + \abs{ \p_x^3 g_{\tau} }_{0; \, E}\frac{ \abs{x}^2 }{2n} \right) - \Delta \left( \p_x^3 u_{\tau} + \abs{ \p_x^3 g_{\tau} }_{0; \, E}\frac{ \abs{x}^2 }{2n} \right) &= \p_t(\p_x^3 u_{\tau}) - \Delta (\p^3_x u_{\tau}) - \abs{ \p_x^3 g_{\tau} }_{0; \, E} \\
	&= \p_x^3 g_{\tau} - \abs{ \p_x^3 g_{\tau} }_{0; \, E} \leq 0,
\end{align*}
and hence the function $\p_x^3 u_{\tau} + \abs{ \p_x^3 g_{\tau} }_{0; \, E(x_0, t_0; r) } \abs{ x }^2 / (2n)$ is subsolution of the heat equation.  From the mean value property of subsolutions we have
\begin{equation}
\p_x^3 u_{\tau}(x,t) \leq \frac{1}{4r^n} \iint_{E(x,t;r)} \left( \p_y^3 u_{\tau}(y,s) + \abs{\p_x^3 g_{\tau}}_{0} \frac{ \abs{ y }^2 }{ 2n } \right) \frac{ \abs{ x - y }^2 }{ \abs{ t - s }^2 } \, dy ds. \label{eqn: heat subsol 4}
\end{equation}
By translating coordinates we can assume without loss of generality that $(x,t) = (0,0)$. Establishing the required estimates involves evaluation of an integral of the form
\[ \frac{1}{r^n}\int_{ \frac{-r^2}{4\pi} } \frac{ R_r(s)^{\alpha} }{ s^{\beta} } \, ds, \]
where $\alpha$ and $\beta$ are given integers.  The constants can be computed explicitly, however we are only interested in the scaling behaviour with respect to the radius $r$ (and that the integral is finite).  We compute
\begin{align*}
	\frac{1}{r^n}\int_{ \frac{-r^2}{4\pi} }^0 \frac{ R_r(s)^{\alpha} }{ s^{\beta} } \, ds &= \frac{1}{r^n}\int_{ \frac{-r^2}{4\pi} }^0 \frac{ \big( -2ns \log[ r^2/ (-4\pi s ) ] \big)^{ \alpha/2} }{ -s^{\beta} } \\
	&= C(n, \alpha, \beta) r^{-n+\alpha-2\beta+2} \int_{ \frac{1}{4\pi} }^0 t^{\alpha/2 - \beta} \big( \log( 4 \pi t) \big)^{\alpha/2} \,dt \\
	&= C(n, \alpha, \beta) r^{-n+\alpha-2\beta+2} \int_0^{\infty} s^{\alpha/2} e^{-\alpha/2-\beta+1} \, ds \\
	&= C(n, \alpha, \beta) r^{-n+\alpha-2\beta+2},
\end{align*}
where the last lines follows from properties of the Gamma function. We point that the integral is independent of $\alpha$ and $\beta$ when $\alpha=\beta=2$.  Returning to \eqref{eqn: heat subsol 4}, we have
\begin{equation}
	\p_x^3 u_{\tau}(x, t)  \leq 4r^{-n} \iint_E \p_y^3 u_{\tau}(y,s) \frac{ \abs{y}^2 }{ s^2 } \, dy ds + \frac{4r^{-n}}{2n}\abs{ \p_x^3 g_{ \tau } }_{0; \, E} \iint_E \frac{ \abs{y}^4 }{ s^2 } \, dy ds. \label{eqn: heat subsol 5}
\end{equation}
We estimate the first term on the right by
\begin{align*}
	4r^{-n} \iint_E \p_y^3 u_{\tau}(y,s) \frac{ \abs{y}^2 }{ s^2 } \, dy ds &\leq Cr^{-n} \int_{ \frac{ -r^2 }{4\pi} }^0 \frac{ R_r(s)^2 }{s^2} \left( \int_{B_{R_r(s)}} \p_y^3 u_{\tau} \, dy \right) ds \\
	&\leq Cr^{-n} \int_{ \frac{ -r^2 }{4\pi} }^0 \frac{ R_r(s)^2 }{s^2} \left( \int_{\p B_{R_r(s)}} \osc \p_y^2 u \, dy \right) ds \\
	&\leq Cr^{-n} \osc_E \p_x^2 u \int_{ \frac{ -r^2 }{4\pi} }^0 \frac{ R_r(s)^{n+1} }{s^2} ds\\
	&\leq \frac{ C(d) }{ r } \osc_E \p_x^2 u.
\end{align*}
The second term on the right of \eqref{eqn: heat subsol 5} can be estimated more simply to give
\[ 4\abs{ \p_x^3 g_{ \tau } }_{0; \, E} r^{-n} \iint_E \frac{ \abs{y}^4 }{ s^2 } \, dy ds \leq C(d) r^2 \abs{ \p_x^3 g_{ \tau } }_{0; \, E}. \]
The other three estimates \eqref{eqn: heat subsol 2} - \eqref{eqn: time subsol 2} follow in a similar way. For example, to derive estimate \eqref{eqn: heat subsol 2}, we differentiate \eqref{eqn: heat subsol 0} twice in space and once in time. Using the mean value property of subsolutions of the heat equation we obtain
\begin{equation}
	\p_t \p_x^2 u_{\tau}(x, t)  \leq 4r^{-n} \iint_E \p_s \p_y^2 u_{\tau}(y,s) \frac{ \abs{y}^2 }{ s^2 } \, dy ds + \frac{4r^{-n}}{2n}\abs{ \p_t \p_x^2 g_{ \tau } }_{0; \, E} \iint_E \frac{ \abs{y}^4 }{ s^2 } \, dy ds. \label{eqn: heat subsol 5}
\end{equation}
The first part of the estimate follows by integrating by parts in time:
\begin{align*}
	4r^{-n} \iint_E \p_t\p_y^2 u_{\tau}(y,s) \frac{ \abs{y}^2 }{ s^2 } \, dy ds &\leq Cr^{-n} \iint_E \osc \p_y^2 u \frac{ \abs{y}^2 }{ s^3 } \\
	&\leq \frac{ C(d) }{ r^2 } \osc_E \p_x^2 u,
\end{align*}
and the second term on right can be estimate in the same way as before to give \eqref{eqn: heat subsol 2}.

The derivation of the parabolic Schauder estimate now continues in the the same way as for the Poisson equation, using the estimates \eqref{eqn: heat subsol 1} - \eqref{eqn: time subsol 2}, the equivalence of norms lemma and the H\"{o}lder space interpolation inequality. After completing these calculations we obtain obtain the desired parabolic Schauder estimate:
\begin{equation}
	[\p^{2,1} u ]_{\alpha; \, Q_R } \leq C \big( [ f ]_{\alpha; \, Q_R }+ \abs{ u }_{0; \, Q_R } \big), \label{eqn: eqn Schauder}
\end{equation}
where $C$ depends on $d, \lambda, \Lambda$, and $\alpha$.

\end{document}